\documentclass[10pt,reqno]{amsart}

\usepackage[cp1251]{inputenc}

\usepackage{amssymb
,epsfig
}
\usepackage{graphicx}

\newtheorem{dfn}{Definition}
\newtheorem{thm}{Theorem}
\newtheorem{prp}{Proposition}
\newtheorem{lemma}{Lemma}

\newtheorem{rem}{Remark}

\def\R{{\mathbb R}}
\def\<{\langle}
\def\>{\rangle}

\begin{document}

\title{Point Charges and Polygonal Linkages}

\author{Giorgi Khimshiashvili*, Gaiane Panina**, Dirk Siersma$^\dag$, Vladimir Zolotov$^\ddag$}

\address{$*$Ilia State University, Tbilisi, Georgia,
e-mail: khimsh@rmi.acnet.ge.
 $**$ Institute for Informatics and Automation, St. Petersburg, Russia,
Saint-Petersburg State University, St. Petersburg, Russia,
e-mail:gaiane-panina@rambler.ru. $^\dag$ University of Utrecht,
Utrecht, The Netherlands, e-mail: D.Siersma@uu.nl. $^\ddag$
Saint-Petersburg State University, St. Petersburg, Russia, e-mail:
paranuel@mail.ru.}


\begin{abstract}
We investigate the critical points of Coulomb potential of point charges placed at the vertices of a planar polygonal linkage.
It is shown that, for a collection of positive charges on a pentagonal linkage, there is a unique critical point in the set of convex
 configurations which is the point of absolute minimum.
 This enables us to prove that two controlling charges are sufficient to navigate between any two convex configurations of a pentagonal linkage.
\end{abstract}

\maketitle
\section{Introduction}

We deal with the Coulomb potential of a system of point charges
placed at the vertices of a planar polygonal linkage. The ultimate
goal is to establish the possibility of controlling the shape of a
pentagonal linkage by varying the values of two charges at the
vertices. Following a paradigm developed in our previous paper
\cite{kps} we consider Coulomb potential as a meromorphic function
on the planar moduli space of linkage and investigate its critical
points. The basic assumption and motivation for accepting such an
approach is that a vertex-charged linkage subject only to Coulomb
interaction of its charged vertices should take the shape with the
minimal Coulomb potential.

The approach developed in \cite{kps} was suggested by
some recent research concerned with the control of nanosystems
and other systems with several degrees of freedom \cite{Jing}, \cite{Wang}.
This setting suggests several aspects and problems. In the present
paper, we concentrate on the following scenario.

Given a convex planar configuration of linkage, we wish to find the
vertex charges such that the global minimum of the arising Coulomb
potential is achieved at the given configuration. Such a collection
of charges will be said to {\it stabilize} the given configuration.
We prove that any convex configuration of pentagonal linkage has a
stabilizing system of charges which depends continuously on the
configuration. This allows one to navigate any (initial) convex
configuration to any other (target) convex configuration along  any
prescribed path in the space of convex configurations. We will refer
to this situation by saying that a pentagonal linkage admits a {\it
complete Coulomb control}.

We wish to add that this research arose as a natural continuation of
our previous joint results on Morse functions on moduli spaces of
polygonal linkages \cite{khipan},  \cite{kpsz}.

\bigskip

The paper is organized as follows.

Section \ref{SectPrelim} contains necessary preliminaries. Among them is  a formula
for the partial derivatives of Cayley-Menger determinant (Theorem \ref{ThmCayleyMenger})
which seems interesting for its own sake and provides one of our main technical tools.
Section \ref{Sectionquadri} briefly sketches an outline of the
proof of the Coulomb control for quadrilateral linkages. In a sense,
it is a prologue which introduces the approach to be used for
pentagonal linkages in the subsequent sections.

In Section \ref{LocalMinSection}  we charge a pentagonal linkage by
a five-tuple of positive charges \newline $q=(q_1,...,q_5)$. We
prove that for any collection of  charges, the  Coulomb potential
has a unique critical point (which is the  minimum point) in the
space of convex configurations.

In Section \ref{pentControlSection}  we control a pentagonal linkage
by just two positive charges. This means that we put charges
$(q_1,q_2,t,q_4,s)$, assuming that $q_1,q_2,q_4$ are some fixed
positive charges, and  that we can vary $s$ and $t$. We prove that  for
any convex configuration $P$  there exists a unique stabilizing pair of positive
charges $s,t$. This yields our main result (Theorem
\ref{ThmMainControl}) stating  that these two charges provide a
complete control on the space of convex configurations.

 To simplify and properly structurize the
presentation all technical proofs are placed in the separate
Section \ref{SectionProofs}.

\textbf{Acknowledgements.} The present paper was completed during a
"Research in Pairs" session in CIRM (Luminy) in January of 2015. The
authors acknowledge the hospitality and excellent working conditions
at CIRM. We also thank Pavel Galashin for useful and
inspiring discussions. G. Panina and V. Zolotov were  supported by
RFBR, research project No. 15-01-02021.

\section{Notation and preliminaries}\label{SectPrelim}
\subsection{Polygonal linkages and their moduli spaces}

A \textit{polygonal linkage} $L$ is  defined by a collection of
positive numbers $l = (l_1,...,l_n)$, called {\it sidelengths},
which we express by writing $L = L(l)$. Physically, a polygonal
linkage is a collection of rigid bars of lengths $l_i$ joined in a
cycle by revolving joints. It is a flexible mechanism which can
admit different shapes,  with or without self-intersections.

By $M(L)$ we denote the \textit{moduli space} of planar
configurations,
 that is, the space of all polygons with the prescribed edge lengths factorized
 by isometries of $\mathbb{R}^2$:

$$M(L)=\{(p_1,...,p_n)| p_i \in \mathbb{R}^2, |p_ip_{i+1}|=l_i, \ |p_np_{1}|=l_n\}/Iso(\mathbb{R}^2).$$

This is not exactly the moduli space $\mathcal{M}(L)$ treated in
\cite{khipan} and \cite{F}, where the space of polygons is
factorized by orientation preserving isometries. However there is a
two-fold covering $\mathcal{M}(L)\rightarrow M(L).$

 By $M^C(L)$ we denote the set of all \textit{strictly convex}  configurations.
 We exclude here non-strictly convex polygons,  that is, those having (at least) one angle
 equal to
$\pi$. The latters obviously form the boundary $\partial M^C(L)$.

 The
set of all  convex configurations, that is,  the closure of $M^C(L)$
in the ambient space $M(L)$, is denoted by
$\overline{M^C}=\overline{M^C}(L)$.

In this paper we basically deal with $n=4, \ 5$. For a 4-bar polygonal
linkage, ${M}(L)$ is a (topological) circle, whereas $\overline{M^C(L)}$ is
homeomorphic to a segment. For a 5-bar polygonal linkage, $M(L)$ is
(generically) a surface, whereas $\overline{M^C}(L)$ is homeomorphic
to a disk $D^2$.

\subsection{Partial derivatives of Cayley-Menger determinant}\label{CMSection}

We present now the definition and certain properties of the
Cayley-Menger determinant which is one of our main tools.

Let $A_1,A_2,A_3,A_4$ be four points in $\R^3$. Denote the distances
between the points and the vectors as

 $$d_{ij}=|A_iA_j|,\ \ \overrightarrow{d}_{ij}=\overrightarrow{A_iA}_j.$$

The Cayley-Menger determinant of the quadruple of points is defined by the formula

$$D =
 \begin{vmatrix}
  0 & 1 & 1 & 1 & 1\\
  1 & 0 & d_{12}^2 & d_{13}^2 & d_{14}^2 \\
  1 & d_{12}^2  & 0 & d_{23}^2 &  d_{24}^2 \\
  1 & d_{13}^2 & d_{23}^2 & 0 & d_{34}^2  \\
  1 & d_{14}^2 & d_{24}^2 & d_{34}^2 & 0 \\
 \end{vmatrix}.$$

As is well known, $D$ is equal to 288 times the squared volume of the tetrahedron  $A_1,...,A_4$.

\bigskip

Denote by $S_{ijk}$  the\textit{ oriented area} of the triangle $A_iA_jA_k$, that is, the orientation is taken into account.
As an example, $S_{ijk}=-S_{ikj}$.

 We also need the vectors $$\overrightarrow{S}_{ijk}=\frac{1}{2}\overrightarrow{d}_{ij}\times \overrightarrow{d}_{jk},$$
 for which we have
 $$|\overrightarrow{S}_{ijk}|=|S_{ijk}|.$$
\begin{thm} \label{ThmCayleyMenger}
\begin{enumerate}
  \item For any four points $A_1,...,A_4$    in $\R^3$, we have:
$$\frac{1}{2}\frac{\partial}{\partial d_{13}^2}D =-16\langle\overrightarrow{S}_{124}, \overrightarrow{S}_{234}\rangle, $$
where $\langle.,.\rangle$ denotes the standard scalar product in the space $\R^3$.
  \item
In the case where the four points  are coplanar, the formula reduces to
$$\frac{1}{2}\frac{\partial}{\partial d_{13}^2}D =-16S_{124}\cdot S_{234}. $$
\end{enumerate}
\end{thm}

\begin{proof}
 $$\frac{1}{2}\frac{\partial}{\partial d_{13}^2}D =
 \begin{vmatrix}
  0 & 1 & 1 &  1\\
  1 & d_{12}^2  & 0  & d_{24}^2  \\
  1 & d_{13}^2 & d_{23}^2 & d_{34}^2  \\
  1 & d_{14}^2 & d_{24}^2 & 0 \\
 \end{vmatrix}
$$
$$
= - \begin{vmatrix}
  0 & 1 & 1 &  1\\
  1 &  0  & d_{12}^2  & d_{24}^2  \\
  1 & d_{23}^2 & d_{13}^2 & d_{34}^2  \\
  1  & d_{24}^2 & d_{14}^2 & 0 \\
 \end{vmatrix} =
$$

$$
=-\begin{vmatrix}
  0 & 1 & 0 & 0 \\
  1 &  0  & d_{12}^2 &  d_{24}^2  \\
  0 & d_{23}^2 & d_{13}^2 - d_{23}^2 - d_{12}^2 & d_{34}^2 - d_{23}^2 - d_{24}^2  \\
  0  & d_{24}^2 & d_{14}^2 - d_{12}^2 - d_{24}^2 & -2d_{24}^2 \\
 \end{vmatrix} =
$$

$$
=\begin{vmatrix}
 d_{13}^2 - d_{23}^2 - d_{12}^2 & d_{34}^2 - d_{23}^2 - d_{24}^2  \\
 d_{14}^2 - d_{12}^2 - d_{24}^2 & -2d_{24}^2 \\
 \end{vmatrix}=
$$

$$
=-4\begin{vmatrix}
\langle \overrightarrow{d}_{12}, \overrightarrow{d}_{23}\rangle & \langle \overrightarrow{d}_{23}, \overrightarrow{d}_{24}\rangle \\
\langle \overrightarrow{d}_{12}, \overrightarrow{d}_{24}\rangle & \langle \overrightarrow{d}_{24}, \overrightarrow{d}_{24}\rangle \\
 \end{vmatrix}=
$$

$$
=-4\langle \overrightarrow{d}_{12}\times \overrightarrow{d}_{24}, \overrightarrow{d}_{23}\times \overrightarrow{d}_{24}\rangle= -16\langle\overrightarrow{S}_{124}, \overrightarrow{S}_{234}\rangle,
 $$
 where ''$\times$''  denotes the cross-product of vectors and the last line of equalities
 follows from the Binet-Cauchy formula.

\end{proof}

This theorem seems to be new.  To the best of our knowledge, this
formula appeared in the literature only for the case when the points
$A_1,...,A_4$ are coplanar. For example, in \cite{CorsRoberts} it is
given without proof but with the following comments. "This is an
important formula concerning the Cayley - Menger determinant. This
formula is only valid when restricting to planar configurations. The
minus sign in equation  is not included in Dziobek's original paper
\cite{Dz} nor in several later works that utilize the Cayley-Menger
determinant. However, checking equation on the square configuration
indicates the need for the minus sign. The correct formula appears
in the doctoral thesis of Hampton \cite{H}".

We actually use it only in the coplanar case  for describing relations between
diagonals of polygons. However we presented the proof of a more general statement for completeness
and convenience of the reader more so that \cite{H} does not seem to be easily accessible.

\bigskip

%%%%%%%%%%%%%%%%%%%%%%%%%%%%%%%%%%%%%%
\subsection{Coulomb potential}

Placing a collection of point charges $q_i$ at the vertices $A_i$ of
a configuration and considering the Coulomb potential of these
charges we get a  function defined on $M(L)$. We will refer to this
setting by speaking of a {\it vertex-charged linkage} with the
system of charges $q =(q_1,...,q_n)$.

Recall that
the Coulomb potential $\tilde{E}$ of a system of point charges $ q_i
\in \mathbb{R}$ placed at the points $p_i$ of Euclidean plane
$\mathbb{R}^2$ is defined as

\begin{equation} \label{coulomb}
\tilde{E} = \tilde{E}_q =\sum_{i\neq j} \frac{q_iq_j}{d_{ij}},
\end{equation}
where $d_{ij} = \vert A_i A_j \vert$ is the distance between the
$i$-th and $j$-th points.

Since we are only interested in critical points of Coulomb
potential, addition of a constant makes no difference. By the very
definition of polygonal linkage,  the distances corresponding to two
consequent vertices in formula (\ref{coulomb}) remain the same for
all configurations of linkage. Hence their sum is constant for any
fixed collection of charges for our purposes it is sufficient to
work with the {\it effective} Coulomb potential $E$ defined as

\begin{equation} \label{linkcoulomb}
E = E_q = \sum_{i<j-1} \frac{q_iq_j}{d_{ij}},
\end{equation}
where $d_{ij}$ is the length of diagonal between (non-neighboring)
$i$-th and $j$-th vertices of the configuration.

We say that a collection of charges \textit{stabilizes} the
configuration $P$ if $E$ attains at $P$ its minimal value. In this
case we say that $P$ is \textit{the minimum point} of $E$.

We explicate now the setting and notation in the cases considered
in the sequel. For $n=4$, we put one
positive charge $t$  at the first vertex. The rest three
vertex-charges are $+1$.

For $n=5$, we put two positive charges $s$ and $t$ at  any two
non-neighboring vertices and say that $s$ and $t$ are
\textit{controlling charges}.
The rest three charges are some fixed positive numbers.

\bigskip

It is proven in \cite{kps} that,
for positive charges $s,\ t$,
the global minimum of $E$ on the moduli space $M(L)$ belongs
to $M^C$.

\section{Coulomb control of quadrilaterals}\label{Sectionquadri}

As a visual illustration of the paradigm  we briefly recall the main
results of \cite{kps}. For a quadrilateral linkage, we put charges
equal to $1$ at three vertices, whereas the fourth vertex is charged by
$t$. Then  we have
$$E=\frac{1}{d_{13}} +\frac{t}{d_{24}}.$$

The following facts were established in \cite{kps}:
\begin{enumerate}
                          \item For a given convex quadrilateral  $P \in M(L)$, there exists a unique $t$ such that $P$
is a critical point for this charge. In this case, $t$ is positive.
                          \item For a quadrilateral linkage,  $E$ has a unique  minimum  in $M^C$.
                        \end{enumerate}

This means that, for a quadrilateral linkage, we have a complete control of convex
configurations by point charges at its vertices.

The "picture of what happens" which is behind the proofs of
\cite{kps} is such. The configuration space of a quadrilateral is a
circle, and therefore the relation between the diagonals is
important since both diagonals appear in   $E$. We remind that four
points $A_1,...,A_4$ are coplanar whenever their Cayley-Menger
determinant vanishes. Therefore the squared diagonals
$x=d_{13}^2,y=d_{24}^2$ of the quadrilateral are related by a cubic
equation $D(x,y)=0$. The equation defines a real elliptic curve
lying in the $(x,y)$ plane. It contains two connected components:
one is unbounded and the other one is a convex closed curve, which
we denote by   $C$. We make use of the mapping $$\Psi: M\rightarrow
C,$$
$$\Psi(P)=(x,y)=(d_{13}^2,d_{24}^2),$$
which maps   a configuration of the linkage to the point in the
plane whose coordinates are squared diagonals. The configuration space is mapped to the
bounded component of the elliptic curve, whereas the space of convex
configuration is bijectively mapped to some arc of the curve $C$.

The rest of the proof from \cite{kps}   uses the idea of convexity. Namely, we can
consider $E$ as a convex function defined on the $(x,y)$-plane. We
are interested in its restriction to the arc of the convex curve
$C$. The restriction of a convex function to a convex curve is (in
general case) not necessarily convex.  But in our case, due to
the specific position of the gradient $\nabla E$ with respect to the
mentioned arc, and with a special parametrization, the restriction is a convex function, and therefore
has a unique critical point which is the minimum.

In Section \ref{LocalMinSection}  we apply a similar approach: we
introduce a mapping $\Phi$ from the configuration space of a
pentagonal linkage to the space $\R^5$  by mapping a configuration
to its squared diagonals. The image is a 2-dimensional manifold with corners, which appears to be the intersection
of convex hypersurfaces.

\section{Critical points  of Coulomb potential}\label{LocalMinSection}

Now we assume that we are given a polygonal 5-bar linkage $L$ with the vertices
denoted by  $A_1,\dots,A_5$, as in the previous section.  We also adopt notations
for $a_i$ and $b_i$ from the previous section. Since it is a linkage, we now assume that
the sidelengths $|A_1A_2| = a_1,\dots,|A_5A_1| = a_5$  are fixed. Let $$E =
\sum_{i<j-1}{\frac{c_{ij}}{|A_iA_j|}}$$  be the  (effective)
Coulomb potential of the positively charged linkage $L$. Here we  denote $c_{ij}:=q_iq_j \ge 0$.

Let $M^C(L)$ be the set of all convex configurations of $L$. The main result of this section
establishes the uniqueness of critical point of $E$ in $M^C(L)$.

\begin{thm} \label{ThmUniqueMin} The potential $E$ has a unique critical point
in the set  $M^C(L)$, which is the (global) minimum of E on $M(L)$.\qed
\end{thm}

The rest of the section is devoted to the proof of this theorem.

%%Is it valid if we speak here of ALL critical points in M^C?

\begin{dfn}A \textit{k-slice} $X_k \subset \overline{M^C}(L)$ with respect to the diagonal $A_3A_5$ is the set of all
(convex) configurations such that $b_4=|A_3A_5| = k$, where $k$ is
some constant.
\end{dfn}

%%Maybe better dia-slice??

Each slice is an analytic curve homeomorphic to either a line
segment or a point. There are exactly two slices, called {\em
terminals} which are points. The (disjoint) union of all the slices
equals the set $\overline{M^C}(L)$.

\begin{prp}
\label{PropSliceMin} {For each of the slices $X_k$, there exist two possibilities:\begin{enumerate}
                                                                             \item The restriction of
the potential  $E|_{X_k}$ has a unique critical point which is the
minimum  point.
                                                                             \item The restriction of
the potential  $E|_{X_k}$ has no critical points. Then the minimum value is achieved at the boundary  of the slice.
                                                                           \end{enumerate}
In any case,
the minimum point of the restricted potential is unique and depends continuously on $k$.}\qed
\end{prp}

\bigskip

Now we pass from one particular slice to the set of all convex configurations.

\begin{dfn} On each of the slices, we mark the  minimum  point of
$E|_{X_k}$. Taken together for all slices the marked points form the
\textit{polar curve} $ \Gamma \subset \overline{M^C}(L)$ of the
potential $E$ with respect to the diagonal $A_3A_5$.
\end{dfn}

\begin{figure}\label{FigGammaI}
\centering
\includegraphics[width=8 cm]{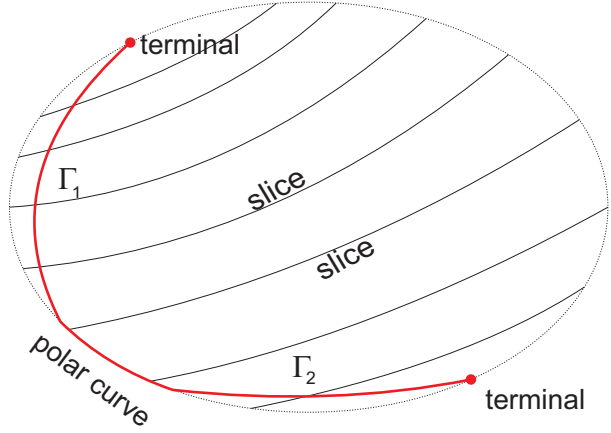}
\caption{The space $ \overline{M^C}(L)$, the polar curve $\Gamma$,
and its components $\Gamma_1$ and $\Gamma_2$. }
\end{figure}

The polar curve is a piecewise analytic curve homeomorphic to a line
segment. Its endpoints are  the terminals. The polar curve may also
contain some segments lying on the boundary of $\overline{M^C}(L)$.
The intersection of the polar curve with the interior part $M^C(L)$
is a finite set of connected components $\Gamma_i$.

Since all critical points of $E$ belong to the polar curve, to prove
Theorem \ref{ThmUniqueMin} it suffices to show that $E|_{\Gamma}$
has a single critical point in $\Gamma$.

\begin{prp} \label{LemmaAlmostThmUniqueMin} The potential $E|_{\Gamma}$ has a unique critical point
in  each of the connected components $\Gamma_i$ of the polar
curve $\Gamma$.\qed
\end{prp}

\bigskip

 With these ingredients at hand, we are ready to prove the main result.

 \subsection*{ Proof of Theorem \ref{ThmUniqueMin}}
 Let us start with  the most symmetric case: 
 the equilateral linkage and equal charges. We prove that  $E$ has a unique
critical point in $\overline{M^C}(L)$.

Indeed, each slice $X_k$ contains an axial symmetric configuration,
that is, there exists a convex polygon $A_1A_2A_3A_4A_5$ such that:

$$|A_3A_5| = k,$$

$$\angle A_2A_3A_4 = \angle A_4A_5A_1, \ \  \angle A_1A_2A_3 = \angle
A_5A_1A_2.$$

By symmetry reasons,  $A_1...A_5$ is a critical point, and
therefore, (the unique)   minimum point of the potential $E$
restricted to $X_k$.
  If $X_k$ is not a terminal  point, the critical point lies in the relative interior of $X_k$.
  Therefore, for this particular case,  the polar curve  lies in the space
  $M^C(L)$, and
  by Proposition \ref{LemmaAlmostThmUniqueMin}
    we obtain the result for the equilateral pentagon and equal charges.

%%%%%%%%%%%%%%%%%%%%%%%%%%%%%%%%%%%%%%%%%%%%%%%%%%
%%\textbf{NEW TEXT,PLease check carefully}
%%%%%%%%%%%%%%%%%%%%%%%%%%%%%%%%%%%%%%%%%%%%%%%%%%%

For an arbitrary pentagonal linkage $L$ and arbitrary (positive)
charges, we use reductio ad absurdum. Assume that $\Gamma$ has more
than one component, and $E$ has more than one critical points. We
start continuously deforming the lengths and the charges aiming to
the above symmetric linkage. Let us inspect the behavior of the
critical points during the deformation. During this process all
objects change continuously: the spaces $M(L)$ and $M^C(L)$, the
slices $X_k$, the curve $\Gamma$ and the critical points. The number
of components of $\Gamma$ can increase, but also decrease (as soon
as the polar curve crosses the boundary).
 In the beginning of the deformation we have
more than one critical points, whereas at the end the critical
point is unique. The confluence of critical points cannot happen in
the interior, since then two branches of $\Gamma$ have to be
connected, but the limits of the critical points are distinct on
that new branch unless they meet at the boundary. However two
critical points on one branch is impossible by Proposition
\ref{LemmaAlmostThmUniqueMin}.

Moreover, it is proven in \cite{kps} that $E$ has no critical points
on the boundary of $\overline{M^C}(L)$. This means that change of
the number of critical points is not possible, so from the beginning
we have only one critical point which is the minimum.
 \qed
%%%%%%%%%%%%%%%%%%%%%%%%%%%%%%%%%%%%%
%\textbf{end careful check}
%%%%%%%%%%%%%%%%%%%%%%%%%%%%%%%%%%

\section{Coulomb control of pentagonal linkages}\label{pentControlSection}

For a generic pentagonal linkage, we put charges $s$ and $t$ to the vertices 5 and 3
respectively.
For this case we have
$$E=\frac{q_1q_4}{d_{14}}+ \frac{q_2q_4}{d_{24}}+ \frac{q_1t}{d_{13}}+ \frac{q_2s}{d_{25}}+ \frac{st}{d_{35}}.$$

%%%%%%%%%%%%%%%%%%%%%%%%%%%%%%%%%%%%
%%\textbf{NEW TEXT,PLease check carefully}
%%%%%%%%%%%%%%%%%%%%%%%%%%%%%%%%%%%
Now we wish to understand whether two charges can provide a complete
control of convex pentagons. We begin with presenting a simple but
conceptually important general observation valid for arbitrary polygonal
linkages. This observation makes essential use of the specific form of Coulomb
interaction and underlies much of the further discussion.

\begin{prp}\label{LemmaGeorge}
For any $n$-gonal linkage $L$ and any configuration $P \in M(L)$,
the stabilizing charges for $P$ are solutions to a system of $n-3$
quadratic equations in $n$ unknowns with the coefficients algebraically
expressible through the lengths $d_{ij}$ of the diagonals of $P$.

\end{prp}

Proof. The proof is obtained by a standard use of the Lagrange
multipliers method for constrained optimization of $E$ as a function
of diagonals $d_{ij}$. In our case the number of variables $d_{ij}$
is equal to $k(n) = n(n-3)/2$. It is also easy to see that the
number of independent constraints $D_j = 0$ given by the
Cayley-Menger relations for the diagonals is equal to $l(n) =
(n-2)(n-3)/2$. Following prescriptions of the Lagrange method we
consider a $(l(n)+1) \times k(n))$ functional matrix $J(P)$ having
the gradient of target function $E_q$ as the first row, and
gradients of constraints $D_j = 0$ as the remaining rows. Notice
that the values of charges $q_j$ appear only in the first row.

According to Lagrange criterion a system of charges $q$ yields a
constrained critical point of $E$ on $M(L)$ if and only if the rank
of $J(P)$ is not maximal. In other words, all \newline $(l(n)+1)
\times (l(n)+1)$-minors of $J(P)$ should vanish at point $P$ which
gives us a system of algebraic equations $S(n)$ for $q$ each of
which is of degree not exceeding two. By linear algebra the number
of independent equations in the system $S(n)$ is equal to $n-3$. The
statement about coefficients can be verified directly.

\qed

For a given $P\in M(L)$, let us denote by $Stab(P)$ the set of
solutions to $S(n)$, i.e. the set of all charges stabilizing $P$.
For obvious reasons, one may await that generically $Stab(P)$ is
three-dimensional and it is easy to see that this is true for $n=4$.
Indeed, in this case we have just one homogeneous quadratic equation
in four unknowns $q_i$ of the form $Aq_1q_3 + Bq_2q_4$. It follows
that the solution set $Stab(P)$ is a cone over the one-sheeted
two-dimensional hyperboloid. We add that similar results can be
obtained for $n\geq 5$ using a general method for geometric and
topological investigation of intersections of real quadrics
developed in \cite{Agr}.

In general, the structure of $Stab(P)$ can vary from point to point,
which makes it unclear how to navigate from one configuration to
another one. Developing effective navigation would be easier if the
sets of available charges were finite. For dimensional reasons the
number of solutions to $S(n)$ may be finite if we reduce the number
of unknowns (charges) to $n-3$. So it becomes natural to suggest
that Coulomb control should be possible if the number of controlling
charges is $n-3$, which will be assumed from now. In this setting
Proposition \ref{LemmaGeorge} yields a modified system $S'(n)$ and a universal upper
estimate for the number of stabilizing charges which follows merely
from Bezout  theorem.

\begin{lemma} \label{Estimate}
If $Stab(P)$ is finite then the number of stabilizing charges for
$P$ does not exceed $2^{n-3}$.
 \qed
\end{lemma}

%Examining the structure of system $S'(5)$ we conclude that if
%$Stab(P)$ is finite then its cardinality in fact does not exceed
%two. The reason is that if the controlling charges are not
%neighboring then the system $S'(n)$ consists of one linear and one
%quadratic equation.

 %Explicit expressions for the coefficients of $S(n)$ in the case of pentagon
%can be found in Section (\ref{SectionProofs}.\ref{TheoremDirkProof}).

Examining the structure of system $S'(5)$ we conclude that if
$Stab(P)$ is finite then its cardinality in fact does not exceed
two. The reason is that if the controlling charges are not
neighboring then the system $S'(5)$ consists of one linear and one
quadratic equation.
In this case explicit expressions for the coefficients
can be found in Section (\ref{TheoremDirkProof}).

 Moreover,
if we only consider convex pentagons with two controlling charges then it turns out that
there exists exactly one pair of positive charges stabilizing $P$, which is
the desired situation for our purposes. To prove the latter fact we need one more lemma.

%%\textbf{end new piece}

\bigskip

\begin{lemma}\cite{kps}\label{LemmaPlusMinus}
For $st<0$, a convex pentagon is never a critical point of
$E$. \qed
\end{lemma}

\begin{thm}\label{TheoremDirk} For each convex configuration $P$ of a pentagonal linkage,
there exists exactly one $(s,t) \in (\R_{>0})^2$ such that $P$ is a
critical point for the charges $(s,t)$.\qed
\end{thm}

%%\textbf{DIRK PLEASE CHECK}
%\textbf{NEW TEXT,PLease check carefully}
%\bigskip
%%%%%%%%%%%%%%%%%%%%%%%%%%%%%%%%%%%%%%%%%%%%%
%\textbf{new text starts here}
%%%%%%%%%%%%%%%%%%%%%%%%%%%%%%%%%%%%%%%%%%%%%%%5

We are now able to explicate the complete control
for convex configurations of arbitrary pentagonal linkage,
which is the main conceptual result of this paper.

\begin{figure}\label{Figstatecontrol}
\centering
\includegraphics[width=10 cm]{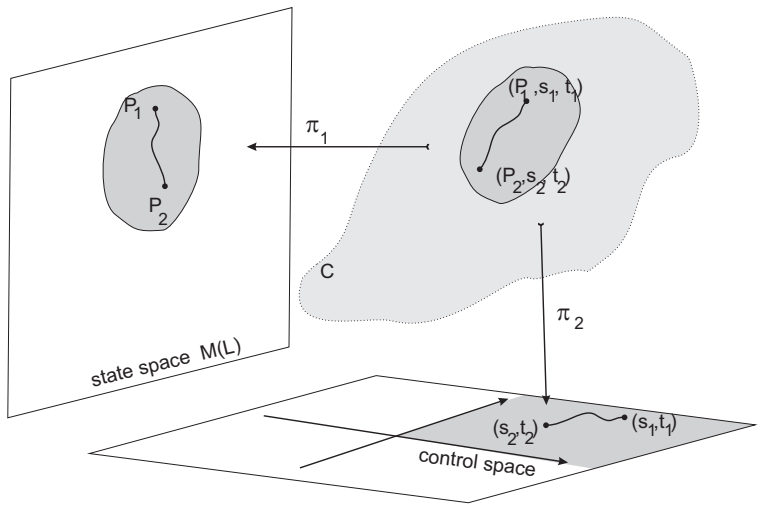}
\caption{The critical space $C$, the state space and the control
space.}
\end{figure}

We describe first our system in the general context of state space
and control space. Our state space is $M(L)$, the space of pentagons
with length vector $L$. The control space is the subspace of charges
 $(q_1,q_2,t,q_4,s)$. So we work with $\mathbb{R}^2 $ as control space.
Given a potential function $$E = E_q : M(L) \times \mathbb{R}^2
\rightarrow \mathbb{R},$$ which maps $(P,q)$ to $E_q(P)$, one can
consider the critical surface $C:= \{(P,q)| \nabla_P E = 0\}.$ We
also consider the subset of minima $C_{min}$.

The restriction to $C$ of the projection to each of the two factors
of $M(L) \times \mathbb{R}^2 $  gives us two maps: $\pi_1: C
\rightarrow M(L)$ and $\pi_2: C \rightarrow \mathbb{R}^2$.
Generically $\pi_2$ is a finite cover which is a local
diffeomorphism away from the bifurcation set in $\mathbb{R}^2$. A
main question is now  if it is possible to connect the two
configuration $P_0$ and $P_1$ by the lift (via $\pi_1$) of a path in
the control space; if possible through the space $C_{min}$. If so,
it is important to construct such a path, which we could call an
explicit control to navigate from $P_0$ to $P_1$.

Theorem  \ref{ThmUniqueMin} about the uniqueness of the absolute
minimum gives an affirmative answer if $P_0$ and $P_1$ are convex.
Indeed $(\pi_2)^{-1}\circ\pi_1$ is a bijection between $M^C (L)$ and
$\R_{>0}^2$.

 The corresponding algorithm to connect  $P_0$ and $P_1$ can be described more
precisely as follows.
First, we calculate the starting and target stabilizing charges $(s_0, t_0)$
and $(s_1, t_1)$  using formulae coming from  the proof of Theorem \ref{TheoremDirkProof} in Section 6.4.
 Next, we choose a path joining $(s_0, t_0)$
and $(s_1, t_1)$ in the control space  $\mathbb{R}_{>0}^2$.
 So if we continuously change the values of stabilizing charges
along the chosen path
the configuration will move to the target configuration through convex configurations
without meeting a bifurcation point.

\bigskip

We summarize the  above as follows.

\begin{thm}\label{ThmMainControl} Assume that a pentagonal linkage $L$   is charged by $(q_1,q_2,t,q_4,s)$, where  $q_1,q_2,q_4$ are some fixed
positive charges.
For any starting convex configuration
$P_0 \in M^C(L)$ together with any target convex configuration $P_1
\in M^C(L)$, there exists an explicit control by two positive ruling
charges $(s,t)$ which yields a path between $P_0$ and $P_1$ lying in
$M^C(L)$.
\end{thm}

\begin{rem}
Since there exist infinitely many paths joining the starting and
target stabilizing charges in the parameter space,  several natural
problems may now be formulated and explored in our setting. In
particular, one can consider various problems in the spirit of
optimal control. For example, since there exist various natural
Riemannian metrics on the configuration space one can investigate
which connecting path in parameter space gives the shortest path in
the configuration space. Thinking of linkage as a device performing
certain task one may wish to specify a path in the working space of
a certain vertex, say, to avoid collision with an obstacle. Further
examples of such problems can be easily formulated and will be
discussed elsewhere.
\end{rem}

\begin{rem}
It seems worthy of noting that placing the controlling charges at
adjacent vertices the situation becomes worse. Namely, we will not be able
to reach all of the convex polygons (see \cite{kps}). For
instance, we will never be able to have two vertices simultaneously
aligned for an equilateral pentagon. By continuity
reasons  an entire neighborhood of such a polygon becomes
unreachable.
\end{rem}

\section{Proofs}\label{SectionProofs}

\subsection{Notation and applications of Cayley-Menger determinant's derivatives }\label{SectionCayMengDerivat}

%%Here, in this section, Coulomb potential does not participate.

We start with an elementary application of Theorem
\ref{ThmCayleyMenger} about Cayley-Menger determinant.

Assume that we are given five points  $A_1,...,A_5$ in the
space $\R^3$. Let us introduce some ad hoc notation convenient for
our purposes:

$$a_i :=|A_iA_{i+1}|,$$
$$b_i:= |A_{i-1}A_{i+1}|,$$
$$x_i=b_i^2,$$

where we treat the indices modulo five.

We think of $a_i$ as the edges and of $b_i$ as diagonals of the
pentagon  $A_1,...,A_5$. Notice that we permit that the pentagon may
be non-planar and non-convex.

We also need the squared diagonals $x_i$  which will play the role
of variables.

Let $D_ i$ be the Cayley-Menger determinant for the vertices of
quadrilateral which is obtained by cutting of a triangle along
$i$-th diagonal. For example, $D_4$ is the Cayley-Menger determinant
for the quadruple  $A_1,A_2,A_3,A_5$:

$$D_4 =
 \begin{vmatrix}
  0 & 1 & 1 & 1 & 1\\
  1 & 0 & a_1^2 & x_2 & a_5^2 \\
  1 & a_1^2  & 0 & a_2^2 & x_1  \\
  1 & x_2 & a_2^2 & 0 & x_4  \\
  1 & a_5^2 & x_1 & x_4 & 0 \\
 \end{vmatrix}.$$

Applying Theorem \ref{ThmCayleyMenger} in this situation we get several useful equalities.

\begin{lemma}\label{DPartLem}
 Assume that $A_1,...,A_5$ are coplanar points. Then we have:
  $$\frac{1}{2}\frac{\partial}{\partial x_2}D_4 =-16S_{125} S_{235}. $$
  $$\frac{1}{2}\frac{\partial}{\partial x_4}D_4 =+16S_{125} S_{123}. $$
 $$\frac{1}{2}\frac{\partial}{\partial x_1}D_4 =-16S_{123} S_{135}. \qed$$
\end{lemma}

\bigskip

\subsection{Proof of Proposition \ref{PropSliceMin}.}

 Throughout the paragraph we fix $b_4$  and  parameterize the slice $X_k$ by $x_2$. In this setting  the
lengths of diagonals $b_i$, the squared lengths of diagonals $x_i$,
and the restriction of the potential $E|_{X_k}$ are the functions in
the variable $x_2$.  All derivatives (denoted by "prime") mean
derivatives with respect to $x_2$. For instance, we write
$(b_1)'=\frac{db_1}{dx_2}.$

\begin{lemma}\label{LemmaAreasVSDerivat}
\begin{enumerate}
\item $$(-x_1)' = \frac{S_{125}}{S_{123}} \cdot \frac{S_{235}}{S_{135}}$$
\item  $$(-x_3)'=\frac{S_{234}}{S_{135}} \cdot \frac{S_{125}}{S_{123}}$$
\item   $$(x_5)'=\frac{S_{145}}{S_{135}}.$$
\end{enumerate}
\end{lemma}

Proof follows from Lemma \ref{DPartLem} by using the implicit
differentiation formula. \qed

\bigskip

\begin{lemma}\label{LemmaLengthsMonot}
For a fixed slice $X_k$, we have  $(b_1)' < 0$, $(b_3)' < 0$,
$(b_5)'
> 0$ on the relative interior of $X_k$.
This implies that $b_1\downarrow$, $b_3\downarrow$, $b_5\uparrow$.
\end{lemma}

Proof follows from Lemma \ref{LemmaAreasVSDerivat}.\qed

\begin{lemma}\label{lastlemma}
For a fixed slice $X_k$, we have
\begin{enumerate}
\item{$-(x_1)'' > 0$}
\item{$-(x_5)'' > 0$}
\item{$\frac{(x_3)'}{(x_1)'} \uparrow$}
\end{enumerate}
\end{lemma}
\begin{proof} (1) follows from second order
relation between the diagonals of quadrilateral, see \cite{kps}.

The statement (2)
can be proven the same way as (1).

By  Lemma \ref{LemmaAreasVSDerivat}, $$\frac{(x_3)'}{(x_1)'} =\frac{S_{234}}{S_{235}}.$$  The statement (3) follows now from direct computation of the derivative.\end{proof}

\begin{lemma}\label{SliceSecDer}For a fixed slice $X_k$, we have
$E|_{X_k}''>0$ on  the relative interior of $X_k$.
\end{lemma}

\begin{proof}
The Coulomb potential for the pentagonal linkage writes as:
$$ E = c_{25}b_1^{-1} + c_{13}b_2^{-1} + c_{24}b_3^{-1} + c_{35}b_4^{-1} + c_{14}b_5^{-1}.$$
The derivative of the potential along the slice $X_k$ writes as:
$$ E|_{X_k}' = c_{25}b_1^{-3}(-x_1)'+ c_{24}b_3^{-3}(-x_3)' - c_{13}b_2^{-3} -  c_{14}b_5^{-3}(x_5)'.$$
Using the fact that $(x_1)' < 0$, we can rewrite this formula as:
$$E|_{X_k}' = (-x_1)'\Big(c_{25}b_1^{-3}+ c_{24}b_3^{-3}\frac{(-x_3)'}{(-x_1)'}\Big) - c_{13}b_2^{-3} -  c_{14}b_5^{-3}(x_5)'.$$
To prove that $E''>0$ it suffices to show that:
$$(-x_1)' > 0,$$
$$\Big(c_{25}b_1^{-3}+ c_{24}b_3^{-3}\frac{(-x_3)'}{(-x_1)'}\Big) > 0,$$
$$(-x_1)'' > 0,$$
$$\Big(c_{25}b_1^{-3}+ c_{24}b_3^{-3}\frac{(-x_3)'}{(-x_1)'}\Big) \uparrow,$$
$$(- c_{13}b_2^{-3}) \uparrow,$$
$$(- c_{14}b_5^{-3}(x_5)') \uparrow.$$
All these statements we know from  Lemma \ref{LemmaLengthsMonot} and
Lemma \ref{lastlemma}.
\end{proof}

\vspace{0.5cm}

Lemma \ref{SliceSecDer} implies  Proposition \ref{PropSliceMin}
straightforwardly since
 $E$ is strictly convex (as the
function in the variable $x_2$) on each of the slices.

\subsection{Proof of Proposition \ref{LemmaAlmostThmUniqueMin}}

As was already mentioned, $M^C(L)$ is a two-dimensional closed disk.
We embed $M^C(L)$ in $\R^5$  by mapping each configuration to the
squared lengths of its diagonals $x_1,\dots,x_5$:
$$\Phi: M(L)\rightarrow \mathbb{R}^5.$$
This mapping sends $M^C(L)$ to its image bijectively so this is
indeed an embedding. The mapping extends by the same rule to the
entire configuration space. However on the entire configuration
space it is not injective.

Now we think of $E$ as of a function defined on $\R^5$.

$$E = \frac{c_{25}}{b_1}+\frac{c_{13}}{b_2}+\frac{c_{24}}{b_3}+\frac{c_{35}}{b_4}+\frac{c_{14}}{b_5}.$$

We will deal with the signs of components of its gradient. As a
matter of fact they all are negative:

$$\nabla E = -\frac{1}{2}\Big{(}
\frac{c_{52}}{(b^2_1)^{\frac{3}{2}}},
\frac{c_{13}}{(b^2_2)^{\frac{3}{2}}},
\frac{c_{24}}{(b^2_3)^{\frac{3}{2}}},
\frac{c_{35}}{(b^2_4)^{\frac{3}{2}}},
\frac{c_{41}}{(b^2_5)^{\frac{3}{2}}} \Big{)} \in [-;-;-;-;-].$$

Here we denoted by $[-;-;-;-;-]$ the set $\R_- \times \R_- \times
\R_- \times \R_- \times \R_-$. In the sequel we use analogous
notations regarding various combinations of signs of expressions in
question.

Let $\gamma(t) =
(\gamma_1(t),\gamma_2(t),\gamma_3(t),\gamma_4(t),\gamma_5(t))$,
where $t \in[0,1]$ be a $C^2$-smooth  curve in $\R^5$.  Later we
shall assume that $\gamma$ is the polar curve but now we consider
just any smooth curve. We have

$$E(\gamma(t)) = \sum{\frac{c_{ij}}{(\gamma_k)^{\frac{1}{2}}}},$$
where the sum is over all triples such that $i+1=k(mod \, 5),\
k+1=j(mod \, 5)$. We denote by prime $'$ the derivative
$\frac{d}{dt}$ and compute

$$E(\gamma(t))' = -\frac{1}{2}\sum{\frac{c_{ij}(\gamma_k)'}{(\gamma_k)^{\frac{3}{2}}}},$$

$$E(\gamma(t))'' = \frac{3}{4}\sum{\frac{c_{ij}[(\gamma_k)']^2}{(\gamma_k)^{\frac{5}{2}}}}-
\frac{1}{2}\sum{\frac{c_{ij}(\gamma_k)''}{(\gamma_k)^{\frac{3}{2}}}}.$$

From this we conclude that if the function
$$-\frac{1}{2}\sum{\frac{c_{ij}(\gamma_k)''}{(\gamma_k)^{\frac{3}{2}}}}=
\<\nabla E, (\gamma)''\>$$ is non-negative on the curve  $\gamma$,
then $E$ has a single critical point on the curve $\gamma$.
%%So do we eventaullyt speak of ANY criticla point or just MINIMUM ?!

We remind that we denote by $D_1$ the Cayley-Menger determinant for
the points $A_2,A_3,A_4,A_5$, denote by $D_2$ the Cayley-Menger
determinant for for the points $A_1,A_3,A_4,A_5$, and so on.

The system
$$D_1=D_2=...=D_5=0$$  defines a surface which contains the image of
$M^C(L)$ under the above described mapping.

Let us consider the function $D_2$ separately.
 It is a polynomial of  degree $3$ in
variables $x_2, x_4,x_5$. It doesn't depend on $x_1, x_3$.

\bigskip

Consider an open 3-arm $R$ with edgelengths $a_3,a_4,a_5$. It is a
subchain of our $5$-linkage. The configuration space $M(R)$ of the arm is homeomorphic to the $2$-torus.

We map the configuration space of the arm  to $\R^3$  by mapping
each configuration to the squared lengths of its three diagonals
$b_4,b_5$ and $b_2$:

$$\Phi_2: M(R) \rightarrow \mathbb{R}^3.$$

The image of $\Phi_2$  belongs to the set $D_2 = 0$.

We denote by $M^C(R)$  those configurations of the arm that
are subconfigurations of some element of $M^C$, that is, that are
extendable to a convex pentagon.

\begin{lemma} \begin{enumerate}
                \item $\Phi_2$ maps $(M^C(R))$  to its image bijectively.
                \item $\Phi_2(M^C(R))$ is a convex surface (with  boundary) in $\R^3$.
                \item The map  $\Phi_2$  on $M^C(R)$ has maximal rank except for aligned configurations.
              \end{enumerate}

\end{lemma}

Proof. The surface $\Phi_2(M(R))$ is a closed surface
contained in $D_2=0$. The surface $\Phi_2(M(R))$ bounds in
$\R^3$ some  body. The image of $\Phi_2$  belongs to the set
$D_2=0$. $D_2$ is a polynomial of degree three, therefore each
generic line  intersects $D_2=0$ at at most three points. Since a
line intersects a closed surface at an even number of points, each
generic line  intersects $\Phi_2(M(R))$ at most at two points.
\qed

Notice that this surface is contained in a "box" in the positive
octant.

\bigskip

From Theorem \ref{ThmCayleyMenger} we know the signs of all entries of
all the $D_i$. For example,

$$\nabla D_2 = (0; \ 16S_{345}S_{145};\ 0;\ -16S_{134}S_{145};\ -16S_{345}S_{135}).$$

We present all these signs in the following table:

 \bigskip

\begin{center}
\begin{tabular}{|l|l|l|l|l|l|}
  \hline
  $\nabla D_1$ & $+$ & 0 & $-$ &$-$ & 0 \\
  \hline
  $\nabla D_2$ & 0 & + & 0 & $-$ & $-$ \\
   \hline
  $\nabla D_3$ & $-$ & 0 & + & 0 & $-$ \\
   \hline
  $\nabla D_
  4$ & $-$ & $-$ & 0 & + & 0 \\
   \hline
$ \nabla D_5$ & 0 & $-$ & $-$ & 0 & + \\
  \hline
\end{tabular}
\end{center}

%Similar
% statements hold for other $D_i$:
% $$\nabla D_1 =
%(+;0;-;-;0), \ \ \ \nabla D_3
% = (-;0;+;0;-),$$ $$\nabla D_4 =
%(-;-;0;+;0),\ \ \ \nabla D_5 = (0;-;-;0;+).$$

\bigskip

\begin{lemma} The gradient
$\nabla D_2$ is the inner normal vector of the surface $D_2=0$.
Similar statements hold for other $D_i$.
\end{lemma}

Proof. It is sufficient to check if the gradient points inside or
outside for one point only. We assume that
 we pick a pentagon without any aligned  edges. We reduce the dimension as follows. First fix $x_2$ and after
 that $x_4$. There are only two quadrilaterals satisfying this condition:
one non-convex and one convex (which has bigger $x_5$). The
intersection with the convex body is an interval. The $x_5$
component of the gradient is there negative, so the gradient vector
points inside. \qed

%Let $\gamma(t) = \gamma(x_4)$ be the (image under $\Phi$ of) polar curve parameterized by $x_4$.

\begin{rem}
The curve  $\Gamma \cap M^C(L)$ is given by  $E'=0$ and therefore is
an algebraic curve. (We remind that $E'$ is the derivative of $E$
along the slice $X_k$ parameterized by $x_2$). Lemma
\ref{SliceSecDer} claims that $E''>0$ on $X_k$. Hence, we conclude
(with the implicit function theorem) that $\Gamma \cap M^C(L)$ is a
smooth curve which intersects $X_k$ transversally.
\end{rem}

\bigskip

Let $\gamma(t) = \gamma(x_4)$ parameterize the image $\Phi(\Gamma)$
of the polar curve.

We are now able to prove Proposition \ref{LemmaAlmostThmUniqueMin}
by establishing the  inequality $$\<\nabla E, \gamma''\> \geq 0
\hbox{ on each of the branches $\Gamma_j$.}$$

Assume that a point $\gamma(t)$ does not lie on the boundary of
$\Phi(M^C)$.
 The curve $\gamma$ lies on
$\Phi(\overline{M^C(L)})$ which is the part of the
 convex component   of $D_i =
0$  for each $i$. As we explained above, this component of $D_i = 0$
is the boundary of some convex body, so by the Darboux formulae for the
normal curvature we have:

$$\<\gamma(t)'';\nabla D_i\> \ge 0$$

%%\subsection{The tangent vector of slice comes into play }

Let now $\zeta $ be a tangent vector of the slice at the point
$\gamma(t)$.  By Lemma \ref{LemmaLengthsMonot} we can assume that
$$\zeta \in [-;+;-;0;+].$$

By the definition of polar curve we have $\<\nabla E, \zeta\> = 0$,
and obviously $\<\nabla D_i, \zeta\> = 0$. To show that $\<\nabla E;
\gamma(t)''\> \ge 0$, it suffices  to show that $\<\nabla E;
\gamma(t)'' + \lambda \zeta \> \ge 0$ for some $\lambda$. For any
$\lambda$, we have already shown that $\<\nabla D_i ; \gamma(t)'' +
\lambda \zeta \> \ge 0$.

%%\subsection{Final step: analysis of signs}

Since $\gamma(t)$ is parameterized by $x_4$, we have $\gamma(t)''
\in [*;*;*;0;*]$, where  $*$ denotes entries of unknown signs. Since
$\<\nabla D_4; \gamma(t)''\> \ge 0$,  only three cases are possible:

$$\gamma(t)'' \in [-;-;*;0;*], \ \  \gamma(t)'' \in [+;-;*;0;*], \ \hbox{  and } \gamma(t)'' \in
[-;+;*;0;*]$$.

We treat these cases separately.
\begin{enumerate}
  \item The first case is simple: since $\<\nabla D_1;
\gamma(t)''\> \ge 0$, $\<\nabla D_2;\gamma(t)''\> \ge 0$  we have  $\gamma(t)'' \in (-;-;-;0;-)$, and
$\<\nabla E; \gamma(t)''\> \ge 0$, since $\nabla E\in(-,-,-,-,-)$.
This completes the proof of Proposition
\ref{LemmaAlmostThmUniqueMin}.
  \item In the second case we  use $\<\nabla D_2; \gamma(t)''\> \ge 0$ and get $\gamma(t) \in (+;-;*;0;-)$. Here we  have two cases: $\gamma(t)'' \in (+;-;+;0;-)$ and $\gamma(t)'' \in (+;-;-;0;-)$.
Assume we have $\gamma(t)'' \in (+;-;+;0;-)$ (the other case is
treated similarly).
 Let us take $\gamma(t)'' + \lambda \zeta$ and look how the signs change when we continuously increase $\lambda$ from $0$ to $+\infty$.
  We start from $(+;-;+;0;-)$ and go to $\zeta \in (-;+;-;0;+)$, which means that all the entries (except for $0$) change their signs.
   Let us enumerate the signs this way: $(+_1;-_2;+_3;0;-_5)$. The inequality $\<\nabla D_4 ; \gamma(t)''+ \lambda \zeta\> \ge 0$ implies that the first sign  changes before the second.
    The inequality $\<\nabla D_1 ; \gamma(t)''+ \lambda \zeta\> \ge 0$
     implies  that the third  sign  changes before the first. $\<\nabla D_2 ; \gamma(t)''+ \lambda \zeta\> \ge 0$
      implies that the second sign  changes before the fifth. So at some moment we necessarily have \newline
      $\gamma(t)'' + \lambda \zeta \in (-;-;-;0;-)$.
       Then $\<\nabla E; \gamma(t)'' + \lambda \zeta \> \ge 0$, which implies \newline $\<\nabla E; \gamma(t)''\> \ge 0$, and we are done.
  \item The third case is treated similarly to the second one.
\end{enumerate}

The proof of Proposition \ref{LemmaAlmostThmUniqueMin} is now
completed.\qed

\subsection{Proof of Theorem \ref{TheoremDirk}.}\label{TheoremDirkProof}

The equilateral case was already proven in \cite{kps}. We follow its proof.
We rewrite potential in the form

$$E=\frac{q_1q_4}{b_5}+\frac{q_2q_4}{b_3}+\frac{q_1t}{b_2}+\frac{q_2s}{b_1}+ \frac{st}{b_4}.$$

Take now the diagonals $b_{4}$ and $b_{2}$ as local coordinates in a
neighborhood of $P$.

The polygon $P$ is a critical point of $E$  means that $dE$
vanishes:
$$-\frac{\partial E}{\partial b_4}=  \frac{q_1q_4\alpha_1}{b_5^2}+ \frac{q_2q_4\beta_1}{b_3^2}+ \frac{sq_2\gamma_1}{b_1^2}+\frac{st}{b_4^2}= 0,
$$
and

$$-\frac{\partial E}{\partial b_2}=  \frac{q_1q_4\alpha_2}{b_5^2}+ \frac{q_2q_4\beta_2}{b_3^2}+ \frac{q_1t}{b_2^2}+\frac{q_2s\gamma_2}{b_1^2}= 0,
$$

where  $$\alpha_1=\partial b_5/
\partial b_4, \
 \beta_1=\partial b_3/ \partial
b_4, \
 \ \gamma_1=\partial b_1/ \partial
b_4, \hbox{  \  and}$$

$$\alpha_2=\partial b_5/
\partial b_2, \
 \beta_2=\partial b_3/ \partial
b_2, \
 \ \gamma_2=\partial b_1/ \partial
b_2 \; .$$

We get a system in two variables $s$ and $t$ which reduces to the
following quadratic equation in $s$:
$$A+Bs+Cs^2=0$$
with
$$A=\frac{q_1q_4\alpha_1}{b_5^2}+\frac{q_2q_4\beta_1}{b_3^2},$$
$$B=\frac{q_2\gamma_1}{b_1^2}-\frac{b_2^2}{b_4^2}(\frac{q_4\alpha_2}{b_5^2}+\frac{q_2q_4\beta_2}{q_1b_3^2}),$$
$$ \hbox{and \ \ \ }C=-\frac{b_2^2\gamma_2q_2}{q_1b_4^2b_1^2}.$$

Since  $AC$ is negative  by Lemma \ref{LemmaAreasVSDerivat},
 the equation has exactly one real positive solution $s$.
By Lemma \ref{LemmaPlusMinus}, $t$ is also positive, which completes
the proof.

\section{Concluding remarks}

By Proposition \ref{LemmaGeorge} for any $n$-gon we get
a system $S(n)$ of quadratic equations for the stabilizing charges.
The existence and structure of real solutions to this
system can be analyzed using topological methods of real algebraic
geometry \cite{Agr}. Finding the number of positive solutions to the
reduced system $S'(n)$ is also possible using methods of \cite{AgrLer}.
So it is still unclear whether a similar Coulomb control is possible for bigger number of edges. However, our expectations are:
(1) the Coulomb potential has a unique critical point (which is the global   minimum) in the domain of convex configurations, 
(2) for a complete control, the  non-ruling  charges should not be put at three consecutive vertices, and 
(3)   a convex configuration may have several collections of positive stabilizing charges.

Theorem \ref{TheoremDirk} implies that we can navigate from any
convex configuration to another convex configuration along any path
joining their stabilizing charges in the space of charges. Since the
space of charges is convex, one can use just the segment joining the
stabilizing charges. It will be interesting to visualize the arising
movement of linkage in the configuration space. This enables one to
consider several natural versions of the optimal control problem for
vertex-charged pentagonal linkages in various contexts.


\begin{thebibliography}{99}

\bibitem{Agr}
A. Agrachev, \emph{Homology of intersections of real quadrics}, Sov. Math. Dokl. 37, 1988, 493-496.

\bibitem{AgrLer}
A.Agrachev, L.Lerano, \emph{Systems of quadratic inequalities}, J. London Math. Soc. 105:3, 2012, 622-660.

\bibitem{conndem} R.Connelly, E.Demaine, \emph{Geometry and topology of polygonal linkages,
Handbook of discrete and computational geometry, 2nd ed.} CRC Press,
Boca Raton, 2004, 197-218.

\bibitem{CorsRoberts}
J.Cors, G.Roberts, \emph{Four-body co-circular central configurations,}
Nonlinearity 25, No.2, 2012, 343-348.

\bibitem{Dz} O. Dziobek,  \emph{Uber einen Merkwurdigen Fall von Vierkoerper problem},
Astron. Nachrichten, 152, 1900, 33-46.

\bibitem{F} M. Farber,  \emph{Invitation to topological robotics}, European Mathematical Society,
2008.

\bibitem{H} M.Hampton,  \emph{Concave central configurations in the four-body problem},
Doctoral Thesis, University of Washington, Seattle, 2002.

\bibitem{Jing}Huang Jing , Li Chuanjiang, Ma Guangfu, Liu Gang, \emph{Coulomb
control of a triangular three-body satellite formation using
nonlinear model predictive method,} Proc. 33rd Chinese Control
Conference (CCC), 2014.

\bibitem{khipan}
G.Khimshiashvili, G.Panina, \emph{Cyclic polygons are critical
points of area},  Zap. Nauchn. Sem. S.-Peterburg. Otdel. Mat. Inst.
Steklov. (POMI), 2008, 360, 8, 238--245.

\bibitem{kps}
G.Khimshiashvili, G.Panina, D.Siersma, \emph{Coulomb control of polygonal linkages},
J. Dyn. Contr. Syst. 14, No.4, 2014, 491-501.


\bibitem{kpsz}
G.Khimshiashvili, G.Panina, D.Siersma, A. Zhukova, \emph{Critical
configurations of planar robot arms},Centr. Europ. J. Math.11(3),
2013, 519--529.

\bibitem{kudernac}T. Kudernac, N. Ruangsupapichat, M. Parschau, B. Mac, N. Katsonis, S.
Harutyunyan, K.-H. Ernst, B. Feringa,  \emph{Electrically
driven directional motion of a four-wheeled molecule on a metal
surface}, Nature 479 (7372): 2011.





\bibitem{Wang} Shuquan Wang and Hanspeter Schaub,  \emph{Coulomb Control of
Nonequilibrium Fixed Shape Triangular Three-Vehicle Cluster},
Journal of Guidance, Control, and Dynamics, Vol. 34, No. 1 (2011),
 259--270.

\end{thebibliography}
\end{document}